\documentclass[12pt]{amsart}

\usepackage{amsrefs}
\usepackage{amssymb}
\usepackage{hyperref}
\usepackage{fancyhdr}
\usepackage{xcolor}
\usepackage{bbm}
\usepackage{url}

\newcommand{\ignore}[1]{}

\newtheorem{theorem}{Theorem}
\newtheorem*{theorem*}{Theorem}
\newtheorem{lemma}[theorem]{Lemma}
\newtheorem{proposition}[theorem]{Proposition}
\newtheorem{claim}[theorem]{Claim}
\newtheorem{corollary}[theorem]{Corollary}
\newtheorem{example}[theorem]{Example}

\newtheorem{conjecture}[theorem]{Conjecture}
\newtheorem{heuristic}[theorem]{Heuristic}
\newtheorem{maintheorem}{Theorem}

\theoremstyle{definition}

\newtheorem*{definition*}{Definition}
\newtheorem*{lemma*}{Lemma}
\usepackage{graphicx}
\usepackage{pstricks, enumerate, pst-node, pst-text, pst-plot}

\numberwithin{equation}{section}
\numberwithin{theorem}{section}

\newcommand{\R}{\mathbb{R}}
\newcommand{\N}{\mathbb{N}}

\newcommand{\eps}{\varepsilon}
\newcommand{\wh}[1]{\widehat{#1}}

\DeclareMathOperator{\sgn}{sgn}
\DeclareMathOperator{\avg}{avg}
\DeclareMathOperator{\aut}{Aut}

\renewcommand{\P}[1]{{\mathbb{P}}\left[{#1}\right]}
\newcommand{\E}[1]{{\mathbb{E}}\left[{#1}\right]}

\newcommand{\CondE}[2]{{\mathbb{E}}\left[{#1}\middle\vert{#2}\right]}

\newcommand{\Esub}[2]{{\mathbb{E}}_{#1}\left[{#2}\right]}
\newcommand{\ind}[1]{{\mathbbm{1}_{\left\{{#1}\right\}}}}

\newcommand{\Stab}{\textnormal{Stab}}
\newcommand{\Maj}[1]{\textnormal{Maj}_{#1}}

\newcommand{\op}{X}

\newcommand{\opz}[1]{\op(#1)}

\newcommand{\ops}[2]{\op_{#2}(#1)}

\newcommand{\mean}[1]{\mu_{#1}}

\newcommand{\neigh}[1]{\partial(#1)}

\newcommand{\wei}{w}

\newcommand{\rgi}{\mathcal{G}_\bullet}

\newcommand{\gnp}[2]{G(#1,#2)}

\newcommand{\rrg}[2]{R(#1,#2)}

\newcommand{\normal}[2]{N(#1,#2)}

\begin{document}

\title[Convergence, unanimity and disagreement]{Convergence, unanimity
  and disagreement in majority dynamics on unimodular graphs and
  random graphs}

\author[I.\ Benjamini]{Itai Benjamini}
\address[I.\ Benjamini]{Faculty of Mathematics and Computer Science, Weizmann Institute of Science.}
\author[S.\ Chan]{Siu-On Chan}
\author[R.\ O'Donnell]{Ryan O'Donnell}
\address[R.\ O'Donnell]{Department of Computer Science, Carnegie Mellon University.}
\author[O.\ Tamuz]{Omer Tamuz}
\address[S.\ Chan, O.\ Tamuz]{Microsoft Research New England.}
\author[L.\ Tan]{Li-Yang Tan}
\address[L.\ Tan]{Department of Computer Science, Columbia University.}

\thanks{R.\ O'Donnell is supported by NSF grants CCF-1319743 and
  CCF-1116594.}

\date{\today}

\begin{abstract}
  In majority dynamics, agents located at the vertices of an undirected
  simple graph update their binary opinions synchronously by adopting
  those of the majority of their neighbors.

  On infinite unimodular transitive graphs (e.g., Cayley graphs), when
  initial opinions are chosen from a distribution that is invariant
  with respect to the graph automorphism group, we show that the
  opinion of each agent almost surely either converges, or else
  eventually oscillates with period two; this is known to hold for
  finite graphs, but not for all infinite graphs.

  On Erd\H{o}s-R\'enyi random graphs with degrees $\Omega(\sqrt{n})$,
  we show that when initial opinions are chosen i.i.d.\ then agents
  all converge to the initial majority opinion, with constant
  probability. Conversely, on random 4-regular finite graphs, we show
  that with high probability different agents converge to different
  opinions.
\end{abstract}

\maketitle
\tableofcontents

\section{Introduction}
Let $G=(V,E)$ be a finite or countably infinite, locally finite,
undirected simple graph. Consider time periods $t \in
\{0,1,2,\ldots\}$ and, for each time $t$ and $i \in V$, let $\ops{i}{t}
\in \{-1,+1\}$ be the opinion of vertex $i$ at time $t$.

We define {\em majority dynamics} by
\begin{align}
  \label{eq:majority-dynamics-def}
  \ops{i}{t+1} = \sgn\sum_{j \in \neigh{i}}\ops{j}{t},
\end{align}
where $\neigh{i}$ is the set of neighbors of $i$ in $G$. To resolve
(or avoid) ties, we either add or remove $i$ from $\neigh{i}$ so that
$|\neigh{i}|$ is odd. This ensures that the sum in the r.h.s.\
of~\eqref{eq:majority-dynamics-def} is never zero. Equivalently, we
let ties be broken by reverting to the agent's existing opinion.

A well known result is the {\em period two property} of finite graphs,
due to Goles and Olivos~\cite{goles1980periodic}.
\begin{theorem}[Goles and Olivos]
  For every finite graph $G=(V,E)$, initial opinions
  $\{\ops{i}{0}\}_{i \in V}$ and vertex $i$ it holds that
  $\ops{i}{t+2} = \ops{i}{t}$ for all sufficiently large $t$.
\end{theorem}
That is, every agent's opinion eventually converges, or else
enters a cycle of length two. 

This theorem also holds for some infinite
graphs~\cite{moran1995period, ginosar2000majority}; in particular for
those of bounded degree and subexponential growth, or slow enough
exponential growth. In~\cite{tamuz2013majority} it is furthermore
shown that on graphs of maximum degree $d$ the number of times $t$ for
which $\ops{i}{t+2} \neq \ops{i}{t}$ is at most
$$\frac{d+1}{d-1} \cdot d \cdot
\sum_{r=0}^\infty\left(\frac{d+1}{d-1}\right)^{-r}n_r(G,i),$$
where $n_r(G,i)$ is the number of vertices at graph distance $r$ from
$i$ in~$G$.

However, on some infinite graphs there exist initial configurations of
the opinions such that no agent's opinion converges to any period;
this is easy to construct on regular trees. A natural question is
whether such configurations are ``rare'', in the sense that they
appear with probability zero for some natural probability distribution
on the initial configurations. In~\cite{kanoria2011majority} it was
shown that on a regular trees, when initial opinions are chosen
i.i.d.\ with sufficient bias towards $+1$, then all opinions converge
to $+1$ with probability one. It was shown also that this is not the
case in some odd degree regular trees, when the bias is sufficiently
small. However, the question of whether opinions converge at all when
the bias is small was not addressed.

We show that indeed opinions almost surely converge (or enter a cycle
with period two) on regular trees, whenever the initial configuration
is chosen i.i.d. In fact, we prove a much more general result.

A {\em graph isomorphism} between graphs $G=(V,E)$ and $G'=(V',E')$
is a bijection $h : V \to V'$ such that $(i,j) \in E$ iff
$(h(i),h(j)) \in E'$. Intuitively, two graphs are isomorphic if they
are equal, up to a renaming of the vertices.

The {\em automorphism group} $\aut(G)$ is the set of isomorphisms from
$G$ to $G$, equipped with the operation of composition. $G$ is said to
be {\em transitive} if $\aut(G)$ acts transitively on $V$. That is, if
there is a single orbit $V/G$, or, equivalently, if for every $i,j \in
V$ there exists an $h \in \aut(G)$ such that $h(i)=j$.  $G$ is said to
be {\em unimodular} if $\aut(G)$ is unimodular (see, e.g., Aldous and
Lyons~\cite{aldous2007processes})\footnote{See~\cite{aldous2007processes}
  for an example (the ``grandfather graph'') of a transitive graph
  that is not unimodular.}. $G$ is unimodular if and only i the
following ``mass transport principle'' holds: informally, in every
flow on the graph that is invariant to $\aut(G)$, the sum of what
flows into a node is equal to the sum of what flows out. Formally, for
every $F \colon V \times V \to \R^+$ that is invariant with respect to
the diagonal action of $\aut(G)$ it holds that
\begin{align*}
  \sum_{j \in \neigh{i}}f(i,j) = \sum_{j \in \neigh{i}}f(j,i),
\end{align*}
where $i \in V$ is arbitrary.

Many natural infinite trasitive graphs are unimodular. These include
all Cayley graphs, all transitive amenable graphs, and, for example,
transitive planar graphs with one end~\cite{lyons-book}.

Our first result is the following.
\begin{maintheorem}[The almost sure period two property for unimodular
  transitive graphs]
  \label{thm:p2p-unimodular}
  Let $G$ be a unimodular transitive graph, and let the agents'
  initial opinions $\{\ops{i}{0}\}_{i \in V}$ be chosen from a
  distribution that is $\aut(G)$-invariant. Then, under majority
  dynamics,
  \begin{align}
    \label{eq:a.s.p2p}
    \P{\lim_t \ops{i}{t+2}-\ops{i}{t} = 0} = 1,
  \end{align}
  and furthermore
  \begin{align*}
    \E{\#\{t\,:\,\ops{i}{t+2} \neq \ops{i}{t}\}} \leq 2d,
  \end{align*}
  where $d$ is the degree of $G$.
\end{maintheorem}
That is, each node's opinion almost surely converges to a cycle of
period at most two. 

In fact, this result is a special case of our
Theorem~\ref{thm:p2p-gen} below, which applies to {\em unimodular
  random networks}. These include many natural random graphs such as
invariant percolation clusters, uniform infinite planar
triangulations~\cite{angel2003uniform} and any limit of finite graphs,
in the sense of~\cite{benjamini2001recurrence}; see
Section~\ref{sec:asp2p} for a formal definition. In fact, this is such
a large family that one may guess that any graph has what we call the
{\em almost sure period two property}: if initial opinions are chosen
i.i.d.\ from the uniform distribution over $\{-1,-1\}$, then each
node's opinion almost surely converges to a cycle of period at most
two.  This, however, is not true, as we show in the next example.
\begin{example}
  \label{thm:not-on-all-infinite-graph}
  There exists an infinite graph $G$ that does not have the almost
  sure period two property.
\end{example}
As a reading of the details of this example will reveal, this graph is
not of bounded degree. We conjecture that
\begin{conjecture}
  Every bounded degree graph has the almost sure period two property.
\end{conjecture}

We next consider the process of majority dynamics on a random finite
graph, where initial opinions $\{\ops{i}{0}\}_{i \in V}$ are chosen
i.i.d.\ from the uniform distribution over $\{-1,+1\}$. Here
convergence to period two is guaranteed by the Goles-Olivos
Theorem. The question we tackle is whether agents all converge to the
{\em same} opinion.

The Erd\H{o}s-R\'enyi graph $\gnp{n}{p}$ is the distribution over graphs
with $n$ vertices in which each edge exists independently with
probability $p$. A random regular graph $\rrg{n}{d}$ is the uniform
distribution over all $d$-regular connected graphs with $n$ vertices.

We first study $\gnp{n}{p_n}$, where $p_n = \Omega(\sqrt{n})$. Following
the usual convention, we say that an event happens {\em with high
  probability} when it happens with probability that tends to one as
$n$ tends to infinity.  Let $\mean{0} = \avg_{i \in V}
\{\ops{i}{0}\}$.
\begin{maintheorem}[Unanimity on high degree Erd\H{o}s-R\'enyi graphs]
  \label{thm:gnp-agreement}
  Assume $n \geq n_0$ and $p \geq cn^{-1/2}$, where $n_0, c > 0$ are
  sufficiently large universal constants. Then with probability at
  least~$.4$ over the choice of $G \sim \gnp{n}{p}$ and the initial
  opinions, the vertices unanimously hold opinion $\sgn(\mean{0})$ at
  time~$4$.
\end{maintheorem}

Next, we consider $\rrg{n}{d}$, with $d=4$. In this setting we prove
the following result. We say that {\em unanimity} is reached at time
$t$ when $\ops{i}{t} = \ops{j}{t}$ for all $i,j \in V$.
\begin{maintheorem}[Disagreement on random regular low degree graphs]
  \label{thm:disagreement}
  Let $G_n$ be drawn from $\rrg{n}{4}$, or be any sequence of
  $4$-regular expanders with growing girth. Choose the initial
  opinions independently with probability $1/3 < p < 2/3$. Then, with
  high probability, unanimity is not reached at any time.
\end{maintheorem}

The following result on finite graphs is an immediate corollary of
Theorem~\ref{thm:p2p-gen}, which is a statement on infinite graphs.
\begin{corollary}
  Let $G$ be drawn from $\rrg{n}{d}$ with $d \geq 3$, or from
  $\gnp{n}{d/n}$ with $d > 1$.

  Then for every $\eps > 0$ there exists a time $t$ such that, with
  high probability, $\ops{i}{t+2} = \ops{i}{t}$ for all $i \in V$
  except a set of size $\eps \cdot |V|$. Furthermore, at this time
  $t$, the fraction of nodes for which $\ops{i}{t} = 1$ is, with high
  probability, in $[1/2-\eps,1/2+\eps]$.
\end{corollary}
Hence at some time $t$ almost all nodes will have already reached
period at most two (at least temporarily), and without having reached
agreement. This, together with the results above, motivates the
following conjecture.
\begin{conjecture}
  Let $G$ be drawn from $\gnp{n}{d_n/n}$.

  \begin{itemize}
  \item When $d_n$ is a bounded, then for every $\eps >0$, with high
    probability, the fraction of nodes for which
    $\lim_t\ops{i}{2t}=+1$ will be in $[1/2-\eps,1/2+\eps]$.

  \item When $d_n \to \infty$, then for every $\eps >0$, with high
    probability, the fraction of nodes for which
    $\lim_t\ops{i}{2t}=+1$ will be in $[0,\eps] \cup [1-\eps,1]$.
  \end{itemize}

\end{conjecture}
That is, stark disagreement is reached for constant degrees, and
unanimity is reached for super-constant degrees. An alternative,
equally reasonable conjecture stipulates that this phase transition
occurs, in fact, when degrees become high enough so that locally the
graph ceases to resemble a tree.

Given a vertex $i$ in a large finite transitive graph and random
uniform initial opinions, consider the Boolean function which is the
eventual opinion of the majority dynamics at $i$, say at even times.
An interesting question is whether this function is local; that is, is
it determined with high probability by the initial opinions in a
bounded neighbourhood of $i$? If it is non-local, can it be
{\em noise-sensitive}~\cite{odonnell12analysis} or it is correlated with the
majority of the initial opinions? Our results so far heuristically
suggest that in the bounded degree regime, majority dynamics is local,
while when the degrees are growing fast enough the majority of the
initial opinions determines the final outcome.  In this respect we
still did not find (or even conjecturally suggest) a family of graphs
in which more interesting global behaviour occurs, such as in
noise-sensitive Boolean functions. Indeed, we are curious to know if
such a family exists.

\section*{Acknowledgments}
The authors would like to thank Microsoft Research New England, where
this research work was substantially performed.

\section{The almost sure period two property}
\label{sec:asp2p}
In this section we shall consider {\em generalized majority dynamics},
or {\em weighted} majority dynamics. In this case we fix a function
$\wei : E \to \R^+$ and let
\begin{align}
  \label{eq:generalized-majority-dynamics-def}
  \ops{i}{t+1} = \sgn\sum_{j \in \neigh{i}}\ops{j}{t} \cdot \wei(i,j).
\end{align}
Note that $\wei(i,j) = \wei(j,i)$, since $\wei$ is a function of the
(undirected) edges. Note also that $\wei(i,i)$ is possibly positive.
We here too assume that $\wei$ is chosen so that the sum in the
r.h.s.\ can never be zero. 

A {\em network} is a triplet $N=(G,\wei,\op)$, where $G=(V,E)$ is a
graph as above, $\op : V \to \{-1,+1\}$ is a labeling of the nodes,
and $\wei : E \to \R^+$ is a weighting of the edges.

In the context of networks, we think of the process of generalized
majority dynamics as a sequence of networks $\{N_t\}$, which all share
the same graph $G_t=G=(V,E)$ and edge weights $\wei_t=\wei$, and where
the node labels $\op_t$ are updated
by~\eqref{eq:generalized-majority-dynamics-def}.

A {\em rooted network} is a pair $(N,i)$ with $N$ a network and $i \in
V$. An isomorphism between two rooted networks $(N,i)$ and $(N',i')$
is a graph isomorphism $h$ between $G$ and $G'$ such that $h(i)=i'$,
$\op = \op' \circ h$ and $\wei = \wei' \circ h$, where we here extend
$h$ to a bijection from $E$ to $E'$. A {\em directed edge rooted
  network} is a triplet $(N,i,j)$ with $(i,j) \in E$. Isomorphisms of
directed edge rooted networks are defined similarly to those of rooted
networks.

A rooted network isomorphism class $[N,i]$ is the set of rooted graphs
isomorphic to $(N,i)$. The set of connected, rooted network isomorphism
classes, which we shall denote by $\rgi$, is equipped with the natural
topology of convergence of finite balls around the root
(see~\cite{benjamini2001recurrence, aldous2004objective}). This
topology provides a Borel structure for probability measures on this
space.

A {\em random network}, or, more precisely, a {\em random rooted
  network isomorphism class} (we shall use the former term), is a
rooted-network-isomorphism-class-valued random variable $[N,I]$; its
distribution is a measure on $\rgi$. Denote by
$\mathcal{G}_{\bullet\bullet}$ the space of isomorphism classes of
directed edge rooted networks $[N,i,j]$. $[N,I]$ is a {\em unimodular
  random network} if, for every Borel $f :
\mathcal{G}_{\bullet\bullet} \to [0,\infty]$, it holds that
\begin{align}
  \label{eq:unimodularity}
  \E{\sum_{j \in \neigh{I}}f(N,I,j)} = \E{\sum_{j \in \neigh{I}}f(N,j,I)}.
\end{align}
We direct the reader to Aldous and Lyons~\cite{aldous2007processes}
for an excellent discussion of this definition.

Let $\{[N_t,I]\}_{t \in \N}$ be a sequence of random networks defined
as follows. Fix some random network $[N_0,I] = [G,\wei,\op_0,I]$. For
$t > 0$, let $[N_t,I] = [G,\wei,\op_t,I]$, where
\begin{align}
  \label{eq:generalized-maj-dyn}
  \ops{i}{t} = \sgn\sum_{j \in \neigh{i}}\ops{j}{t-1} \cdot \wei(i,j).
\end{align}
This sequence of random networks is coupled to share the same (random)
graph, weights and root; only the labeling of the nodes $\op_t$
changes with time. We say that such a sequence is {\em related by
  generalized majority dynamics}.  We impose the condition that $\wei$
is such that almost surely no ties occur (i.e., the sum
in~\eqref{eq:generalized-maj-dyn} is nonzero).

\begin{claim}
  If $[N_0,I]$ is a unimodular random network then so is $[N_t,I]$,
  for all $t \in \N$.
\end{claim}
This follows immediately from the fact that the majority dynamics map
$(G, \wei, \op_{t-1}) \mapsto \ops{i}{t}$ given
by~\eqref{eq:generalized-maj-dyn} is indeed a function of the
rooted network isomorphism class $[N_{t-1},i] \in \rgi$.

For $W,\eps >0$ we say that (the weights $\wei$ of) a random network
$[N,I]$ is {\em $(\eps,W)$-regular} if the following two conditions
hold. First, we require that
\begin{align*}
  \E{\sum_{j \in \neigh{I}}\wei(I,j)} \leq W.
\end{align*}
Note that in the case that $\wei$ is the constant function one, this
is equivalent to having finite expected degree. Next, we require that
\begin{align*}
  \min_{x \in \{-1,+1\}^{\neigh{I}}}\Bigg\lvert\sum_{j \in
    \neigh{I}}\ops{j}{t}\wei(i,j)\Bigg\rvert \geq \eps
\end{align*}
almost surely.  This is an ``ellipticity'' condition that translates
to requiring that one is always $\eps$-far from a tie. In the case
that $\wei$ is the constant function one and degrees are odd, this
holds with $\eps=1$.

We are now ready to state our main result of this section, which is a
generalization of Theorem~\ref{thm:p2p-unimodular} from a fixed
unimodular graph setting to a unimodular random network setting.
\begin{maintheorem}
  \label{thm:p2p-gen}
  Let $\{[N_t,I]\}$ be a sequence of $(\eps,W)$-regular, unimodular
  random networks related by majority dynamics. Then
  \begin{align}
    \label{eq:a.s.p2p-general}
    \P{\lim_t \ops{I}{t+2}-\ops{I}{t} = 0} = 1,
  \end{align}
  and furthermore
  \begin{align*}
    \E{\#\{t\,:\,\ops{I}{t+2} \neq \ops{I}{t}\}} \leq \frac{2W}{\eps}.
  \end{align*}
\end{maintheorem}

Before proving this theorem, we show that it implies
Theorem~\ref{thm:p2p-unimodular}. When the underlying graph of a
random network is a fixed transitive unimodular graph, and when the
distribution of the labels $\ops{i}{t}$ is invariant to the
automorphism group of this graph, then this random network is a
unimodular random network~\cite{aldous2007processes}. Furthermore,
since majority dynamics is generalized majority dynamics with weights
$1$, this random network is $(1,d)$-regular, where $d$ is the
degree. Hence Theorem~\ref{thm:p2p-unimodular} follows.

\subsection{Proof of Theorem~\ref{thm:p2p-gen}}

In this section we prove Theorem~\ref{thm:p2p-gen}. Our proof follows
the idea of the proof of the period two property for finite graphs by
Goles and Olivos~\cite{goles1980periodic}.

Let $\{[N_t,I]\}$ be a sequence of finite expected weighted degree,
unimodular random networks related by majority dynamics.
  
Define the function $f : \mathcal{G}_{\bullet\bullet} \to [0,\infty]$
by
\begin{align*}
  f(N,i,j) = \wei(i,j)\left(1+\op_j\sgn\sum_{k \in \neigh{i}}\wei(i,k)\op_k\right),
\end{align*}
where $N = (G,\wei,\op)$ is a network and $(i,j)$ is an edge in
$G$. If $[N,I]$ is unimodular then
\begin{align}
  \label{eq:mtp1}
  \E{\sum_{j \in \neigh{I}}f(N,I,j)} = \E{\sum_{j \in \neigh{I}}f(N,j,I)}.
\end{align}
Note that $\ops{i}{t+1} = \sgn\sum_{k \in \neigh{i}}\wei(i,k)\op_k$,
and so
\begin{align*}
  f(N_t,I,j) = \wei(I,j)\left(1+\ops{I}{t+1}\ops{j}{t}\right)
\end{align*}
and
\begin{align*}
  f(N_t,j,I) = \wei(I,j)\left(1+\ops{j}{t+1}\ops{I}{t}\right).
\end{align*}
Hence we can write~\ref{eq:mtp1} for $N_t$ as
\begin{align}
  \label{eq:mtp2}
  \E{\sum_{j \in \neigh{I}}\wei(I,j)\ops{I}{t+1}\ops{j}{t}} = \E{\sum_{j \in \neigh{I}}\wei(I,j)\ops{j}{t+1}\ops{I}{t}}.
\end{align}

Next, we define a ``potential''
\begin{align*}
  \ell_t = \frac{1}{4}~\E{\sum_{j \in \neigh{I}}\wei(I,j)\left(\ops{I}{t+1}-\ops{j}{t}\right)^2}.
\end{align*}
Note that $\ell_t$ is positive for all $t$, and also that it is finite
for all $t$, since it is bounded from above by $W$, as a consequence
of the $(\eps,W)$-regularity of $\wei$.

We would like to show that $\ell$ is non-increasing. By definition,
\begin{align*}
  \ell_t - \ell_{t-1} &= -\frac{1}{2}~\E{\sum_{j \in
      \neigh{I}}\wei(I,j)\ops{I}{t+1}\ops{j}{t}}+\frac{1}{2}~\E{\sum_{j \in
      \neigh{I}}\wei(I,j)\ops{I}{t}\ops{j}{t-1}}.
\end{align*}
By~\eqref{eq:mtp2} we can, in the expectation on the right, switch the
roles of $I$ and $j$. Rearranging, we get 
\begin{align*}
  \ell_t - \ell_{t-1} &=
  -\frac{1}{2}~\E{\big(\ops{I}{t+1}-\ops{I}{t-1}\big)\sum_{j \in
      \neigh{I}}\wei(I,j)\ops{j}{t}}.
\end{align*}
Now, $\ops{I}{t+1} = \sgn\sum_{j \in \neigh{I}}\wei(I,j)\ops{j}{t}$,
and so
\begin{align*}
  \lefteqn{\big(\ops{I}{t+1}-\ops{I}{t-1}\big)\sum_{j \in
    \neigh{I}}\wei(I,j)\ops{j}{t}} \\
&= |\ops{I}{t+1}-\ops{I}{t-1}|\cdot\Bigg\rvert\sum_{j \in
      \neigh{I}}\wei(I,j)\ops{j}{t}\Bigg\lvert\\
&= \ind{\ops{I}{t+1} \neq \ops{I}{t-1}}\cdot\Bigg\rvert\sum_{j \in
      \neigh{I}}\wei(I,j)\ops{j}{t}\Bigg\lvert.
\end{align*}
Hence
\begin{align*}
  \ell_t - \ell_{t-1} &=
  -\frac{1}{2}~\E{\ind{\ops{I}{t+1} \neq \ops{I}{t-1}}\cdot\Bigg\rvert\sum_{j \in
      \neigh{I}}\wei(I,j)\ops{j}{t}\Bigg\lvert},
\end{align*}
and we have shown that $\ell_t$ is non-increasing.

Now, by the $(\eps,W)$-regularity of $\wei$ we have that
\begin{align*}
  \Bigg\rvert\sum_{j \in \neigh{I}}\wei(I,j)\ops{j}{t}\Bigg\lvert \geq \eps,
\end{align*}
and so
\begin{align*}
  \ell_t - \ell_{t-1} \leq -\frac{1}{2}~\P{\ops{I}{t+1} \neq
    \ops{I}{t-1}}\cdot \eps.
\end{align*}
Since $\ell_1 \leq W$, and since $\ell_1 \geq
\sum_{t=2}^\infty\ell_{t-1}-\ell_t = \ell_1 - \lim_t\ell_t$, we can
conclude that
\begin{align*}
  \sum_{t=2}^\infty\P{\ops{I}{t+1} \neq
    \ops{I}{t-1}} \leq \frac{2W}{\eps}. 
\end{align*}
Hence
\begin{align*}
  \E{\#\{t\,:\,\ops{I}{t+2} \neq \ops{I}{t}\}} < \frac{2W}{\eps},
\end{align*}
and by the Borel-Cantelli lemma
\begin{align*}
  \P{\lim_t \ops{I}{t+2}-\ops{I}{t} = 0} = 1.
\end{align*}
This completes the proof of Theorem~\ref{thm:p2p-gen}.

\subsection{Example~\ref{thm:not-on-all-infinite-graph}: an infinite
  graph without the almost sure period two property}

Consider an infinite, locally finite graph defined as follows. 
Divide the set of nodes into ``levels'' $L_1,L_2,\ldots$, where level
$L_n$ has $2^n-1$ vertices. Connect each node in $L_n$ with each of
the nodes in $L_{n-1}$, $L_n$ and $L_{n-1}$, except for the nodes in
$L_0$, which are connected only to $L_1$. It follows that
\begin{itemize}
\item Every pair of nodes in the same level have the same set of
  neighbours.
\item \label{item:majority} The majority of the neighbors of $i \in L_n$ are in $L_{n+1}$.
\end{itemize}
Therefore, for all $n$ and for all $i,j \in L_n$, it holds that
$\ops{i}{1} = \ops{j}{1}$. By induction, it follows that $\ops{i}{t} =
\ops{j}{t}$ for all $t \geq 1$, and we accordingly denote
$\ops{L_n}{t} = \ops{i}{t}$ for some $i \in L_n$. Furthermore,
$\ops{L_n}{t} = \ops{L_{n+1}}{t-1}$ for $t \geq 2$, and so
$\ops{L_0}{t} = \ops{L_{t+1}}{1}$ for $t \geq 2$. Finally,
$\{\ops{L_{3n}}{1}\}_{n \in \N}$ are independent random variables, each
uniformly distributed over $\{-1,+1\}$. Hence so are the random
variables $\{\ops{L_0}{3t-1}\}_{t \geq 1}$, and the single node in $L_0$
(and in fact all the other nodes too) does not converge to period two.

\section{Majority dynamics on $\gnp{n}{p}$}

\subsection{Heuristic analysis for the high degree
  case}
\label{sec:heuristic}

Herein we describe a ``heuristic'' analysis suggesting what should
happen for majority dynamics in $\gnp{n}{d_n/n}$ when $d_n =
\omega(1)$ is sufficiently large.  We suggest the reader keep in mind
the parameter range $d_n = n^{\delta}$ where $0 < \delta < 1$ is an
absolute constant.  Our heuristic reasoning will suggest that
unanimity is reached at time roughly $1/\delta + O(1)$.
Unfortunately, we will only be able to make some of this reasoning
precise in the case that $\delta \geq 1/2$. That case is handled
formally in Section~\ref{sec:gnp-proof}.

The \emph{global mean} at time~$t$ is defined to be $\mean{t} =
\avg_{i \in V} \{\ops{i}{t}\}$.  To analyze convergence to unanimity
we will track the progression of $\mean{t}^2$ over time.  The quantity
is nonnegative and it is easy to estimate it initially:
\begin{proposition}                                     \label{prop:initial-mean-sq}
    $\E{\mean{0}^2} = \frac1n$.
\end{proposition}
\noindent On the other hand, we also have $\mean{t}^2 \leq 1$ with equality if and only if there is unanimity at time~$t$.\\

We suggest the following heuristic:\footnote{We here use the notation
  $A \gtrsim B$  for $A = \Omega(B)$.}
\begin{heuristic}                                       \label{heur:mean-square}
    In $\gnp{n}{d_n/n}$, assuming $d = d_n = \omega(1)$ is sufficiently large, we expect $\mean{t+1}^2 \gtrsim d  \mean{t}^2$, provided $d \mean{t}^2 \leq 1$.
\end{heuristic}
Granting this heuristic, we expect the sequence $\mean{0}^2, \mean{1}^2, \mean{2}^2, \dots, \mean{t}^2$ to behave (up to constant factors) as $\frac1n, \frac{d}{n}, \frac{d^2}{n}, \dots, \frac{d^t}{n}$ until $d^t \approx n$. Once $d^t$ is within a constant factor of~$n$ we expect to reach near-unanimity in one more step, and to reach perfect unanimity after an additional step.  For these reasons, we suggest that for $d_n = n^{\delta}$, one may expect convergence to unanimity after $\frac{1}{\delta} + O(1)$ steps.  Our intuition for how well this heuristic should hold when $\delta$ is ``subconstant'' is not very strong, but perhaps it indeed holds so long as $d_n = \omega(1)$.  \\

The remainder of this section is devoted to giving some justification for Heuristic~\ref{heur:mean-square}.  Let us suppose that we have reached time~$t$ and that $d \mean{t}^2 \ll 1$. Computing just the expectation we have
\[
    \E{\mean{t+1}^2} = \avg_{i,j\in V} \E{\ops{i}{t+1}\ops{j}{t+1}}  \approx \avg_{i \neq j} \E{\ops{i}{t+1}\ops{j}{t+1}}.
\]
Here the approximation neglects the case $i = j$; this only affects the average by an additive quantity on the order of~$\frac1n$, which is negligible even compared to~$d \mean{1}^2$.  In a random graph drawn from $\gnp{n}{d/n}$ we expect all pairs of distinct vertices $i,j$ to behave similarly, so we simply consider $\E{\ops{i}{t+1}\ops{j}{t+1}}$ for some \emph{fixed} distinct $i,j \in V$.

Here we come to the weakest point in our heuristic justification; we imagine that the neighbors of~$i$ and~$j$ are ``refreshed'' --- i.e., that we can view them as chosen anew from the $\gnp{n}{d/n}$ model.  For simplicity, we also assume that $i$ and $j$ both have exactly~$d$ neighbors (an odd number).  We might also imagine that they have roughly $\frac{d^2}{n}$ neighbors in common, though we won't use this.  Under these assumptions we have
\[
    \E{\ops{i}{t+1}\ops{j}{t+1}} = \E{\sgn(R_1 + \cdots + R_d)\sgn(S_1 + \cdots + S_d)}
\]
where $R_1, \dots, R_d$ are independent $\{-1,+1\}$-valued random
variables with $\E{R_i} = \mean{t}$, the same is true of $S_1, \dots,
S_d$, and we might assume that some $\frac{d^2}{n}$ of the $R_i$'s and
$S_i$'s are identical.  In any case, by the FKG Inequality (say), we
have
\begin{multline*}
    \E{\sgn(R_1 + \cdots + R_d)\sgn(S_1 + \cdots + S_d)} \\
       \geq \E{\sgn(R_1 + \cdots + R_d)}\E{\sgn(S_1 + \cdots + S_d)}.
\end{multline*}
Thus to finish our heuristic justification of $\mean{t+1}^2 \gtrsim d\mean{t}^2$ it suffices to argue that
\begin{equation}                \label{eqn:finish-heuristic}
    \lvert\E{\sgn(R_1 + \cdots + R_d)}\rvert \gtrsim \sqrt{d}\lvert\mean{t}\rvert.
\end{equation}
Without loss of generality we assume $\mean{t} \geq 0$.  By the Central Limit Theorem, $R_1 + \cdots + R_d$ is distributed essentially as $Z \sim {\normal{d\mean{t}}{d(1-\mean{t}^2)}} \approx \normal{d\mean{t}}{d}$.  (We are already assuming $d \mean{t}^2 \ll 1$, so $\mean{t}^2 \ll 1$ as well.) By the symmetry of normal random variables around their mean we have
\[
    \E{\sgn(Z)} = \P{0 \leq Z \leq 2\E{Z}} = \P{-\sqrt{d}\mean{t} \leq Z' \leq \sqrt{d}\mean{t}},
\]
where $Z'$ is a standard normal random variable.  This last quantity is asymptotic to $\sqrt{\frac{2}{\pi}} \cdot \sqrt{d}\mean{t}$ assuming $\sqrt{d}\mean{t} \ll 1 \iff d \mean{t}^2 \ll 1$, ``confirming''~\eqref{eqn:finish-heuristic}

\subsection{Constant time to unanimity for the very high degree case}
\label{sec:gnp-proof}

In this section we give a precise argument supporting the heuristic analysis from Section~\ref{sec:heuristic} in the case of $\gnp{n}{p_n}$ when $p = p_n \gg 1/\sqrt{n}$.  The main task is to analyze what happens at time~$1$; after that we can apply a result from~\cite{mossel2013majority}, relying on the fact that a random graph is a good expander.  For simplicity we will assume~$n$ is odd, so that $\mean{t}$ is never~$0$.

\begin{proposition}                                     \label{prop:time-1-mean}
    (Assuming $n$ is odd,) $\E{\sgn(\mean{0})\mean{1}} \geq \frac{2}{\pi}\sqrt{p} - \frac{1}{n\sqrt{p}}$.
\end{proposition}
\begin{proof}
    We have $\E{\sgn(\mean{0})\mean{1}} = \avg_{i \in V}[\sgn(\mean{0})\ops{i}{1}]$ and by symmetry the expectation is the same for all~$i$.  Let's therefore compute it for a fixed $i \in V$; say, $i = n$.  Now suppose we condition on vertex~$n$ having exactly~$d$ neighbors when the graph is chosen from~$\gnp{n}{p}$.  The conditional expectation does not depend on the identities of these neighbors; thus we may as well assume they are vertices $1, \dots, d$.  Writing $\opz{j} = \ops{j}{0}$ for brevity, we  therefore obtain
    \begin{gather}
        \E{\sgn(\mean{0})\mean{1}} = \sum_{d = 0}^{n-1} \Pr[\text{Bin}(n-1,p) = d] \times \qquad \qquad \nonumber\\
         \qquad \E{\Maj{n}(\opz{1}, \dots, \opz{n}) \Maj{d'}(\opz{1}, \dots, \opz{d}, \opz{n})}. \label{eqn:maj-level-1}
    \end{gather}
    Here $d'$ denotes $d$ when $d$ is odd and $d+1$ when $d$ is even, and $\Maj{k}(x_1, \dots, x_{\ell})$ denotes $\sgn(x_1 + \cdots + x_k)$.  We can lower-bound the expectation in line~\eqref{eqn:maj-level-1} using Fourier analysis; by Parseval's identity,
    \[
        \eqref{eqn:maj-level-1} = \sum_{S \subseteq [n]}\wh{\Maj{n}}(S) \wh{\Maj{d'}}(S).
    \]
    By symmetry, the value of $\wh{\Maj{k}}(S)$ only depends on~$|S|$; furthermore, it's well known that the sign of $\wh{\Maj{k}}(S)$ depends only on~$|S|$ and not on~$k$~\cite{odonnell12analysis}.  Thus all summands above are nonnegative so we obtain
    \[
        \eqref{eqn:maj-level-1} \geq \sum_{|S| = 1}\wh{\Maj{n}}(S) \wh{\Maj{d'}}(S).
    \]
    Finally, for odd~$k$ we have the explicit formula $\wh{\Maj{k}}(S) = \frac{2}{2^k} \binom{k-1}{\frac{k-1}{2}} \geq \frac{\sqrt{2/\pi}}{\sqrt{k}}$ for any $|S| = 1$. Since the two majorities have exactly~$d'$ coordinates in common, we conclude
    \[
        \eqref{eqn:maj-level-1} \geq d' \frac{\sqrt{2/\pi}}{\sqrt{n}} \frac{\sqrt{2/\pi}}{\sqrt{d'}} = \frac{2}{\pi}\sqrt{\frac{d'}{n}} \geq \frac{2}{\pi}\sqrt{\frac{d}{n}}.
    \]
    Putting this into the original identity we deduce
    \[
        \E{\sgn(\mean{0})\mean{1}} \geq \frac{2}{\pi} \frac{1}{\sqrt{n}} \E{\sqrt{\text{Bin}(n-1,p)}}.
    \]
    We have the standard estimates\footnote{For the first inequality
      see, e.g.,
      \url{http://mathoverflow.net/questions/121411/expectation-of-square-root-of-binomial-r-v}.}
    \[
        \E{\sqrt{\text{Bin}(n-1,p)}} \geq \sqrt{(n-1)p} - \frac{1}{2\sqrt{(n-1)p}} \geq \sqrt{np} - 1.5/\sqrt{np}.
    \]
    Thus we finally obtain
    \[
        \E{\sgn(\mean{0})\mean{1}} \geq \frac{2}{\pi} \sqrt{p} - \frac{1}{n\sqrt{p}},
    \]
    as claimed.
\end{proof}

\begin{proposition}                                     \label{prop:time-1-variance}
     We have
     \[
        \E{(\sgn(\mean{0})\mean{1})^2} = \E{\mean{1}^2} \leq p + \tfrac{3}{pn}.
     \]
\end{proposition}
\begin{proof}
    We have $\E{\mean{1}^2} = \frac{1}{n} + \avg_{i \neq j}\{\E{\ops{i}{1}\ops{j}{1}}\}$; by symmetry it therefore certainly suffices to show
    \begin{equation} \label{eqn:time-1-2nd}
        \E{\ops{i}{1}\ops{j}{1}} \leq p + \tfrac{2}{pn}
    \end{equation}
    for some fixed pair of vertices $i \neq j$. Let us condition on the neighborhood structure of vertices~$i$ and~$j$.  Write $\opz{j} = \ops{j}{0}$ as in the proof of Proposition~\ref{prop:time-1-mean}, and write
    $N'_1 = \neigh{i} \setminus \{j\}$, $N'_2 = \neigh{j} \setminus \{i\}$.  Then
    \[
        \E{\ops{i}{1}\ops{j}{1}} = \E{\Maj{}((\opz{k})_{k \in N_1}) \cdot \Maj{}((\opz{k})_{k \in N_2})}
    \]
    for some sets $N_1' \subseteq N_1 \subseteq N_1 \cup \{i, j\}$ and similarly~$N_2$.  Writing $M = N_1 \cap N_2$ and also $\Maj{N} = \Maj{}(\opz{r} : k \in N)$ for brevity, the above is equal to
    \begin{align}
        &\Esub{N_1, N_2, (\opz{k})_{k \in M}}{\CondE{\Maj{N_1}}{(\opz{k})_{k \in M}} \cdot \CondE{\Maj{N_2}}{(\opz{k})_{k \in M}}} \nonumber\\
        \leq\ &\sqrt{\Esub{N_1, N_2, (\opz{k})_{k \in M}}{\CondE{\Maj{N_1}}{(\opz{r})_{k \in M}}^2}} \nonumber\\
        &\quad \times \sqrt{\Esub{N_1, N_2, (\opz{k})_{k \in M}}{\CondE{\Maj{N_2}}{(\opz{r})_{k \in M}}^2}} \nonumber\\
        =\ &\Esub{N_1, N_2, (\opz{k})_{k \in M}}{\CondE{\Maj{N_2}}{(\opz{r})_{k \in M}}^2}, \label{eqn:bound-me}
    \end{align}
    where the inequality is Cauchy--Schwartz and the final equality is by symmetry of~$i$ with~$j$.

    To analyze~\eqref{eqn:bound-me}, suppose we condition on $N_1$ and
    $N_2$ (hence also $M$).  By symmetry, the conditional expectation
    depends only on $|N_1| = n_1$ and $|M| = m$; by elementary Fourier
    analysis~\cite{odonnell12analysis} it equals
    \begin{align*}
        \sum_{S \subseteq [m]} \wh{\Maj{n_1}}(S)^2 &= \sum_{S \subseteq [n_1]} \frac{\binom{m}{|S|}}{\binom{n_1}{|S|}} \wh{\Maj{n_1}}(S)^2 \\ &\leq \sum_{S \subseteq [n_1]} \left(\frac{m}{n_1}\right)^{|S|} \wh{\Maj{n_1}}(S)^2 = \Stab_{\frac{m}{n_1}}[\Maj{n_1}].
    \end{align*}
    Finally, we have the bounds~\cite{o2003computational}
    \begin{align}
        \Stab_{\frac{m}{n_1}}[\Maj{n_1}] &\leq \frac{m}{n_1}, \nonumber\\
        \Stab_{\frac{m}{n_1}}[\Maj{n_1}] &\leq \frac{2}{\pi}\arcsin\frac{m}{n_1} + O\left(\frac{1}{\sqrt{1-(m/n_1)^2}\sqrt{n}}\right). \label{eqn:arcsin}
    \end{align}
    Although the second bound here would save us a factor of
    roughly~$\frac{2}{\pi}$, for simplicity we'll only use the first
    bound.  It yields
    \[
        \E{\ops{i}{1}\ops{j}{1}} \leq \E{\frac{|M|}{|N_1|}}.
    \]
    Each vertex in $N_1 \setminus \{j\}$ has an (independent) probability~$p$ of being in~$M$; as for~$j$, we'll overestimate by assuming that if $j \in N_1$ then it is always in~$M$ as well.  This leads to
    \[
         \E{\frac{|M|}{|N_1|}} \leq p + \E{\frac{1}{|N_1|}}.
    \]
    Finally, recall that $|N_1|$ is distributed as $\text{Bin}(n-1,p)$
    rounded up to the nearest even integer. Thus (see, e.g.,~\cite{chao1972negative})
    \[
        \E{\frac{1}{|N_1|}} \leq \E{\frac{1}{(\text{Bin}(n-1,p)+1)/2}} = \frac{2}{pn}(1-(1-p)^n) \leq \frac{2}{pn}.
    \]
    This completes the proof.
\end{proof}

\begin{proposition}                                 \label{prop:1-sided-cheb}
    Assume $n \geq n_0$ and $p \geq \frac{c}{\sqrt{n}}$, where $n_0, c > 0$ are sufficiently large universal constants. Then ${\P{\sgn(\mean{0})\mean{1} \geq .006\sqrt{p}} \geq .4004}$.
\end{proposition}
\begin{proof}
    Write $W = \sgn(\mean{0})\mean{1}$.  The ``one-sided Chebyshev inequality'' implies that
    \[
        \P{W \geq .01 \E{W}} \geq \frac{.99^2}{\frac{\E{W^2}}{\E{W}^2} + .99^2 - 1}.
    \]
    Combining Propositions~\ref{prop:time-1-mean}, \ref{prop:time-1-variance}, we have
    \[
        \frac{\E{W^2}}{\E{W}^2} \leq \frac{p + 3/(pn)}{(\frac{2}{\pi}\sqrt{p} - 1/(n\sqrt{p}))^2} \leq \frac{\pi^2}{4} + O\left(\tfrac{1}{p^2n}\right).
    \]
    As $\frac{.02}{\pi} > .006$ and $\frac{.99^2}{\pi^2/4 + .99^2 - 1} > .4004$, the claim follows.
\end{proof}

For good expander graphs of degree~$d$, the results of~\cite{mossel2013majority} show that unanimity will be reached quickly if the global mean ever significantly exceeds~$1/\sqrt{d}$ (in magnitude).  In our situation, we essentially have degree-$pn$ graphs with a constant chance of global mean exceeding~$\Omega(\sqrt{p})$ at time~$1$.  Consequently we are able to show convergence to unanimity provided $p \gg 1/\sqrt{n}$.  

We'll need the following result, which is essentially Proposition~6.2
from~\cite{mossel2013majority} (but slightly modified since we need
not have perfectly regular graphs):
\begin{lemma}                                       \label{lem:mnt-expander}
    Assume $G = (V,E)$ satisfies the following form of the ``Expander Mixing Lemma'': for all $A, B \subseteq V$,
    \[
        \Bigl\lvert E(A,B) - p\lvert A \rvert \lvert B \rvert \Bigr\rvert \leq \lambda \sqrt{\lvert A \rvert \lvert B \rvert},
    \]
    where $E(A,B)$ denotes $\#\{(u,v) \in E : u \in A, v \in B\}$.  Then if majority dynamics on~$G$ ever has $\mean{t} \geq |\alpha|$ then 
    \[
        \#\{i \in V : \ops{i}{t+1} = -\sgn(\mean{t})\} \leq \frac{2\lambda^2}{\alpha^2 p^2 n}.
    \]
\end{lemma}

To use this, we'll also need the following claim that follows easily
from a result of Vu on the spectra of $\gnp{n}{p}$~\cite{vu2005spectral}.
\begin{lemma}
  For $G \sim \gnp{n}{p}$ with $p \gg (\log n)^4/n$ and $n$ high
  enough, the Expander Mixing Lemma holds for~$G$ with $\lambda =
  4\sqrt{np}$, except with probability at most $o(1)$.
\end{lemma}
\begin{proof}
  Let $P_0$ be the (random) adjacency matrix of $G$, and let $P$ be
  given by $P = P_0+D$, where $D$ is a random diagonal matrix whose
  each diagonal entry is one with probability $p$ and zero
  otherwise. $P$ can be thought of as the adjacency matrix of a graph
  $G'$ which is obtained from $G$ by adding each self-loop with
  probability $p$.

  Let $Q$ be the $n \times n$ matrix whose entries are all equal to
  $p$.  Then, since $p \gg (\log n)^4/n$, by~\cite{vu2005spectral} it
  holds with high probability that
  \begin{align*}
    | P-Q | \leq 3\sqrt{np},
  \end{align*}
  where $| \cdot |$ is here the $L^2$ operator (spectral)
  norm. Equivalently, for any two vectors $v,w \in \R^n$,
  \begin{align}
    \label{eq:spectral-norm}
    |v^\top(P-Q)w| \leq 3\sqrt{np} \cdot |v||w|.
  \end{align}

  Let $A,B \subset V$ be any two subsets of vertices. Then the number
  of edges between $A$ and $B$ is given by
  \begin{align*}
    E(A,B) = 1_A^\top P_01_B = 1_A^\top(P-D)1_B = 1_A^\top P1_B-1_A^\top D1_B.
  \end{align*}
  Now, $1_A^\top D1_B$ is at most $|A \cap B|$ and therefore, for $n$ high
  enough it holds that $1_A^\top D1_B \leq \sqrt{np|A||B|}$.
  
  By~\eqref{eq:spectral-norm},
  \begin{align*}
    |1_A^\top P1_B - 1_A^\top Q1_B| \leq 3\sqrt{np}|1_A||1_B| = 3\sqrt{np|A||B|}.
  \end{align*}
  Since $1_A^\top Q1_B = p|A||B|$, with high probability
  \begin{align*}
    \Big\lvert E(A,B) - p|A||B| \Big\rvert =
    |1_A^\top P1_B - 1_A^\top D1_B - 1_A^\top Q1_B|
    \leq 4\sqrt{np|A||B|}.
  \end{align*}
  \qedhere
\end{proof}
Combining the previous two lemmas with Proposition~\ref{prop:1-sided-cheb} we obtain:
\begin{proposition}                                     \label{prop:almost-done-high-degree}
    Assume $n \geq n_0$ and $p \geq \frac{c}{\sqrt{n}}$, where $n_0, c > 0$ are sufficiently large universal constants. Then with probability at least $.4003$ we have
    \[
        \#\{i \in V : \ops{i}{2} \neq \sgn(\mean{0})\} \leq \frac{c}{p^2}.
    \]
\end{proposition}
In $\gnp{n}{p}$ (with $p \gg 1/\sqrt{n}$, say), almost surely each vertex has degree at least~$(p/2)n$, which in turn exceeds~$2c/p^2$ if~$p > (4c/n)^{1/3}$.  Thus we may conclude:
\begin{theorem}    \label{thm:very-high-degree-in-time-3}
    Assume $n \geq n_0$ and $p \geq cn^{-1/3}$, where $n_0, c > 0$ are sufficiently large universal constants. Then with probability at least~$.4$ over the choice of $G \sim \gnp{n}{p}$ and the initial opinions, the vertices unanimously hold opinion $\sgn(\mean{0})$ at time~$3$.
\end{theorem}
In case $n^{-1/2} \ll p \lesssim n^{-1/3}$ we need an extra time period.  By assuming $p \gg n^{-1/2}$, the right-hand side in Proposition~\ref{prop:almost-done-high-degree} can be made smaller than any desired positive constant.  Then applying the two lemmas again we obtain:
\begin{proposition}                                     \label{prop:really-almost-done-high-degree}
    Assume $n \geq n_0$ and $p \geq \frac{c}{\sqrt{n}}$, where $n_0, c > 0$ are sufficiently large universal constants. Then with probability at least $.4002$ we have
    \[
        \#\{i \in V : \ops{i}{3} \neq \sgn(\mean{0})\} \leq \frac{c}{p}.
    \]
\end{proposition}
Now we can finish as in the case of $p \gg n^{-1/3}$; we get:
\begin{theorem*}[\ref{thm:gnp-agreement}]
    Assume $n \geq n_0$ and $p \geq cn^{-1/2}$, where $n_0, c > 0$ are sufficiently large universal constants. Then with probability at least~$.4$ over the choice of $G \sim \gnp{n}{p}$ and the initial opinions, the vertices unanimously hold opinion $\sgn(\mean{0})$ at time~$4$.
\end{theorem*}

As a final remark, when $p = o(1)$ we can (with slightly more effort) improve the probability bound of~$.4$ to any constant smaller than $2/\pi \approx .6366$ by using~\eqref{eqn:arcsin}.

\section{Majority dynamics on $\rrg{n}{d}$}

\begin{proposition}
  \label{prop:disagreement}
  Let $G_n$ be drawn from $\rrg{n}{4}$, or be a sequence of
  $4$-regular expanders with growing girth. Choose the initial
  opinions independently with probability $1/3 < p < 2/3$. Then, with
  high probability, unanimity is not reached at any time.
\end{proposition}
Note that Theorem~\ref{thm:disagreement} is rephrasing of
this proposition.
\begin{proof}
  Consider a growing sequence of $d$-regular expanders with girth
  growing to infinity, denoted by
  $G_n$. By~\cite{alon2004percolation}, $p$-Bernoulli percolation will
  contains a unique giant component of size proportional to $G_n$,
  provided $p > 1/(d-1)$. The same holds for random $d$-regular
  graphs.  In particular if $d=4$ and $ 1/3 < p < 2/3$ an open giant
  component and a closed giant component will coexist.

  Since $d=4$ is even, we take majorities over the four neighbors and
  the vertex itself, as we explain above.  To show that with high
  probability unanimity is not reached on these graphs, it is enough
  then to show that the giant component of $1/2$-Bernoulli percolation
  contains cycles. The opinions on the cycles will not change in the
  process of majority dynamics, since each node will have three
  (including itself) neighbors on the cycle with which it agrees.

  To see this, perform the percolation in two stages: first carry out
  $(p - \epsilon)$-Bernoulli percolation (such that $1/3 < p-\epsilon$),
  and then sprinkle on top of it an independent $\epsilon$-Bernoulli
  percolation. If the first percolation already contains a cycle we
  are done. Otherwise the giant component is a tree.

  Pick an edge of the random giant tree that splits the tree to two
  parts, so that each part has size at least $1/4$ of the tree.
  Denote these two parts by $A$ and $B$.  As in the uniqueness proof
  of~\cite{alon2004percolation}, since $G_n$ is an expander and $A$
  and $B$ has size proportional to the size of $G_n$. there are order
  $\Theta(n)$ disjoint paths of length bounded by a function depending only
  on the expansion.  Thus the $\epsilon$-sprinkling connects $A$ and
  $B$ with order $\Theta(n)$ disjoint open paths, creating many cycles with
  probability tending to $1$ with $n$, and we are done.
\end{proof}

\bibliography{more_maj}

\begin{bibdiv}
\begin{biblist}

\bib{aldous2007processes}{article}{
      author={Aldous, David},
      author={Lyons, Russell},
       title={Processes on unimodular random networks},
        date={2007},
        ISSN={1083-6489},
     journal={Electronic Journal of Probability},
      volume={12},
       pages={no. 54, 1454\ndash 1508},
         url={http://dx.doi.org/10.1214/EJP.v12-463},
      review={\MR{2354165 (2008m:60012)}},
}

\bib{aldous2004objective}{incollection}{
      author={Aldous, David},
      author={Steele, J.~Michael},
       title={The objective method: probabilistic combinatorial optimization
  and local weak convergence},
        date={2004},
   booktitle={Probability on discrete structures},
      series={Encyclopaedia Math. Sci.},
      volume={110},
   publisher={Springer},
     address={Berlin},
       pages={1\ndash 72},
      review={\MR{2023650 (2005e:60018)}},
}

\bib{alon2004percolation}{article}{
      author={Alon, Noga},
      author={Benjamini, Itai},
      author={Stacey, Alan},
       title={Percolation on finite graphs and isoperimetric inequalities},
        date={2004},
        ISSN={0091-1798},
     journal={The Annals of Probability},
      volume={32},
      number={3A},
       pages={1727\ndash 1745},
         url={http://dx.doi.org/10.1214/009117904000000414},
      review={\MR{2073175 (2005f:05149)}},
}

\bib{angel2003uniform}{article}{
      author={Angel, Omer},
      author={Schramm, Oded},
       title={Uniform infinite planar triangulations},
    language={English},
        date={2003},
        ISSN={0010-3616},
     journal={Communications in Mathematical Physics},
      volume={241},
      number={2-3},
       pages={191\ndash 213},
         url={http://dx.doi.org/10.1007/s00220-003-0932-3},
}

\bib{benjamini2001recurrence}{article}{
      author={Benjamini, Itai},
      author={Schramm, Oded},
       title={Recurrence of distributional limits of finite planar graphs},
        date={2001},
        ISSN={1083-6489},
     journal={Electronic Journal of Probability},
      volume={6},
       pages={no. 23, 13 pp. (electronic)},
         url={http://dx.doi.org/10.1214/EJP.v6-96},
      review={\MR{1873300 (2002m:82025)}},
}

\bib{chao1972negative}{article}{
      author={Chao, Min-Te},
      author={Strawderman, William},
       title={Negative moments of positive random variables},
        date={1972},
     journal={Journal of the American Statistical Association},
      volume={67},
      number={338},
       pages={429\ndash 431},
}

\bib{ginosar2000majority}{article}{
      author={Ginosar, Yuval},
      author={Holzman, Ron},
       title={The majority action on infinite graphs: strings and puppets},
        date={2000},
     journal={Discrete Mathematics},
      volume={215},
      number={1-3},
       pages={59\ndash 72},
}

\bib{goles1980periodic}{article}{
      author={Goles, Eric},
      author={Olivos, Jorge},
       title={Periodic behaviour of generalized threshold functions},
        date={1980},
        ISSN={0012-365X},
     journal={Discrete Mathematics},
      volume={30},
      number={2},
       pages={187\ndash 189},
  url={http://www.sciencedirect.com/science/article/pii/0012365X80901211},
}

\bib{kanoria2011majority}{article}{
      author={Kanoria, Yashodhan},
      author={Montanari, Andrea},
       title={Majority dynamics on trees and the dynamic cavity method},
        date={2011},
     journal={The Annals of Applied Probability},
      volume={21},
      number={5},
       pages={1694\ndash 1748},
}

\bib{lyons-book}{book}{
      author={Lyons, Russell},
      author={Peres, Yuval},
       title={Probability on trees and networks},
        date={2013},
}

\bib{moran1995period}{article}{
      author={Moran, Gadi},
       title={On the period-two-property of the majority operator in infinite
  graphs},
        date={1995},
     journal={Transactions of the American Mathematical Society},
      volume={347},
      number={5},
       pages={1649\ndash 1667},
}

\bib{mossel2013majority}{article}{
      author={Mossel, Elchanan},
      author={Neeman, Joe},
      author={Tamuz, Omer},
       title={Majority dynamics and aggregation of information in social
  networks},
        date={2013-06},
     journal={Autonomous Agents and Multi-Agent Systems},
         url={http://dx.doi.org/10.1007/s10458-013-9230-4},
}

\bib{o2003computational}{thesis}{
      author={O'Donnell, Ryan},
       title={Computational applications of noise sensitivity},
        type={Ph.D. Thesis},
        date={2003},
}

\bib{odonnell12analysis}{book}{
      author={O'Donnell, Ryan},
       title={Analysis of {B}oolean functions},
   publisher={Cambridge Univesity Press},
        date={2014},
}

\bib{tamuz2013majority}{article}{
      author={Tamuz, Omer},
      author={Tessler, Ran~J},
       title={Majority dynamics and the retention of information},
        date={2013},
     journal={arXiv preprint arXiv:1307.4035},
}

\bib{vu2005spectral}{article}{
      author={Vu, Van},
       title={Spectral norm of random matrices},
        date={2007},
        ISSN={0209-9683},
     journal={Combinatorica},
      volume={27},
      number={6},
       pages={721\ndash 736},
         url={http://dx.doi.org/10.1007/s00493-007-2190-z},
}

\end{biblist}
\end{bibdiv}
\end{document}